\documentclass{article}

\usepackage{hyperref,pictexwd,amsmath,amssymb,amsthm}
\def\vec{{\bf{v}}}
\newcommand{\End}{\mathrm{End}}
\newcommand{\CC}{\mathbb{C}}  
\newcommand{\id}{\mathrm{id}}  
\newcommand{\rank}{\mathrm{rank}}

\newtheorem{thm}{Theorem}[section]

\newtheorem{prop}[thm]{Proposition}
\newtheorem{cor}[thm]{Corollary}

\theoremstyle{definition}

\usepackage{epsfig}
\usepackage[mathscr]{eucal}
\DeclareMathAlphabet\eurm{U}{eur}{m}{n}
\SetMathAlphabet\eurm{bold}{U}{eur}{b}{n}

\def\upp#1{\leavevmode\raise.31ex\hbox{#1}}
\title{The planar rook algebra and Pascal's  triangle}

\author{Daniel Flath, Tom Halverson, and Kathryn Herbig
\thanks{Halverson and Herbig were supported in part by National Science Foundation Grant DMS-0100975.}}
 
\begin{document}

\maketitle


Surely the best known recursively defined integers are the binomial coefficients $\binom{n}{k}=\binom{n-1}{k-1} + \binom{n-1}{k}$ appearing in Pascal's triangle.
They admit many interpretations, the two most well known of which are 
\begin{enumerate}
\item Algebraic:   Expansion of $(x+y)^n$ with recursion  $$(x+y)^n =( x+y)^{n-1}(x+y).$$
\item Combinatorial:  Number of subsets of $[1, n]= \{ 1, 2, \ldots, n\}$  with recursion $$[1, n]= [1, n-1] \cup \{n\}.$$
\end{enumerate}

In this paper we propose a linear algebra interpretation rooted in representation theory.   We construct natural vector spaces $V^n_k$ with $\dim V^n_k= \binom{n}{k}$ and direct sum decompositions $V^n_k = V^{n-1}_{k-1}\oplus V^{n-1}_k.$
The vector spaces are natural in the sense that the $V^n_k$ are the spaces of all the distinct irreducible representations of an algebra $\CC P_n$, the {\em planar rook algebra}.  The direct sums describe decomposition upon restriction arising from an embedding $\CC P_{n-1}\rightarrow \CC P_n$.  Even the multiplicative structure of the binomial coefficients arises from the representation theory of the planar rook algebras, as we discover upon decomposition of tensor product representations.

The planar rook algebra is an example of a ``diagram algebra," which for our purposes is a finite-dimensional algebra with a basis given by a collection of diagrams and multiplication described combinatorially by diagram concatenation.  When the basis diagrams can be drawn without edge crossings, we get a planar  algebra.  There is a growing theory of planar algebras initiated by V. Jones (see \cite{Jo})  that uses a more refined definition of planarity than what we give here.   

The main goal of this paper is to work out the combinatorial representation theory of the planar rook algebra $\CC P_n$ and to show that it is governed by the theory of binomial coefficients.  
The following are the main results:
\begin{enumerate}
\item A classification of the irreducible $\CC P_n$ modules (Theorems  \ref{ModelDecomposition}  and  \ref{thm:regularrepresentation}).
\item  An explicit decomposition of the regular representation of $\CC P_n$ into a direct sum of irreducibles (Theorem \ref{thm:regularrepresentation}).
\item A computation of the Bratelli diagram for the tower of algebras $\CC P_0 \subseteq \CC P_1 \subseteq \CC P_2 \subseteq \cdots$ (Section 4). 
\item A computation of the irreducible characters, Theorem  \ref{thm:chars}, for $\CC P_n$. 
\end{enumerate}

\section{The Planar Rook Monoid}

Let $R_n$ denote the set of $n \times n$ matrices with entries from $\{0,1\}$ having {\it at most} one 1 in each row and in each column.   For example, the set of all matrices in $R_2$ is
$$
R_2 = \left\{  
\begin{pmatrix} 0 & 0 \\ 0 & 0 \\ \end{pmatrix},
\begin{pmatrix} 1 & 0 \\ 0 & 0 \\ \end{pmatrix},
\begin{pmatrix} 0 & 1 \\ 0 & 0 \\ \end{pmatrix},
\begin{pmatrix} 0 & 0 \\ 1 & 0 \\ \end{pmatrix},
\begin{pmatrix} 0 & 0 \\ 0 & 1 \\ \end{pmatrix},
\begin{pmatrix} 1 & 0 \\ 0 & 1 \\ \end{pmatrix},
\begin{pmatrix} 0 & 1\\ 1& 0 \\ \end{pmatrix}
\right\}.
 $$
We call these ``rook matrices," since the 1s  correspond to the possible placement of non-attacking rooks on an $n \times n$ chessboard.  The {\it rank} of a rook matrix is the number of 1s in the matrix, and so to construct a rook matrix of rank $k$, we choose $k$ rows and $k$ columns in $\binom{n}{k}^2$ ways and then we place the 1s in $k!$ ways.  Thus the cardinality of $R_n$ is given by
$$
\vert R_n \vert = \sum_{k=0}^n  \binom{n}{k}^2 k!.
$$
We let $R_0 = \{ \emptyset \}$. 
There is no known closed formula for the sequence $\vert R_n\vert, n\ge 0$,  which begins $1, 2, 7, 34, 209, 1546, 13327,  \ldots$ (see \cite[A002720]{OEIS}).  The set $R_n$ contains the identity matrix and is closed under matrix multiplication, so $R_n$ is a  monoid (a set with an associative binary operation and an identity, but where elements are not necessarily invertible). The invertible matrices in $R_n$ are the permutation matrices (having rank $n$); they form a subgroup isomorphic to the symmetric group $S_n \subseteq R_n$.  

We associate each element of $R_n$ with a {\it rook diagram}, which is a graph on two rows of $n$ vertices, such that vertex $i$ in the top row is connected to vertex $j$ in the bottom row if and only if the corresponding matrix has a 1 in the $(i,j)$-position.  For example in $R_6$ we have
$$
{\beginpicture
\setcoordinatesystem units <0.4cm,0.18cm>         
\setplotarea x from 1 to 4, y from 0 to 3    
\linethickness=0.5pt                          
\put{$\bullet$} at 1 -1 \put{$\bullet$} at 1 2
\put{$\bullet$} at 2 -1 \put{$\bullet$} at 2 2
\put{$\bullet$} at 3 -1 \put{$\bullet$} at 3 2
\put{$\bullet$} at 4 -1 \put{$\bullet$} at 4 2
\put{$\bullet$} at 5 -1 \put{$\bullet$} at 5 2
\put{$\bullet$} at 6 -1 \put{$\bullet$} at 6 2
\plot 1 2 2 -1  /
\plot 2 2 4 -1  /
\plot 4 2 3 -1  /
\plot 5 2 6 -1  /
\endpicture}
\quad \leftrightarrow \quad
\begin{pmatrix} 
0 & 1 & 0 & 0 & 0 & 0  \\
0 & 0 & 0 & 1 & 0 & 0  \\
0 & 0 & 0 & 0 & 0 & 0  \\
0 & 0 & 1 & 0 & 0 & 0  \\
0 & 0 & 0 & 0 & 0 & 1  \\
0 & 0 & 0 & 0 & 0 & 0  \\
 \end{pmatrix}.
$$
Matrix multiplication is accomplished on diagrams $d_1$ and $d_2$ by placing $d_1$ above $d_2$ and identifying the vertices in the bottom row of $d_1$ with the corresponding vertices in the top
row of $d_2$ (i.e., connecting the dots).   For example,
$$ 
\begin{array}{l}
d_1 = {\beginpicture
\setcoordinatesystem units <0.3cm,0.15cm>         
\setplotarea x from 1 to 5, y from 0 to 3    
\linethickness=0.5pt                          
\put{$\bullet$} at 1 -1 \put{$\bullet$} at 1 2
\put{$\bullet$} at 2 -1 \put{$\bullet$} at 2 2
\put{$\bullet$} at 3 -1 \put{$\bullet$} at 3 2
\put{$\bullet$} at 4 -1 \put{$\bullet$} at 4 2
\put{$\bullet$} at 5 -1 \put{$\bullet$} at 5 2
\plot 1 -1 3 2  /
\plot 2 -1 4 2  /
\plot 4 -1 2 2  /
\plot 3 -1 5 2 /
\endpicture}\\
d_2 = {\beginpicture
\setcoordinatesystem units <0.3cm,0.15cm>         
\setplotarea x from 1 to 5, y from 0 to 3    
\linethickness=0.5pt                          
\put{$\bullet$} at 1 -1 \put{$\bullet$} at 1 2
\put{$\bullet$} at 2 -1 \put{$\bullet$} at 2 2
\put{$\bullet$} at 3 -1 \put{$\bullet$} at 3 2
\put{$\bullet$} at 4 -1 \put{$\bullet$} at 4 2
\put{$\bullet$} at 5 -1 \put{$\bullet$} at 5 2
\plot 2 -1  1 2 /
\plot 3 -1 4 2  /
\plot 4 -1 5 2 /
\endpicture}
\end{array}
\ = \ \ 
{\beginpicture
\setcoordinatesystem units <0.3cm,0.15cm>         
\setplotarea x from 1 to 5, y from 0 to 3    
\linethickness=0.5pt                          
\put{$\bullet$} at 1 -1 \put{$\bullet$} at 1 2
\put{$\bullet$} at 2 -1 \put{$\bullet$} at 2 2
\put{$\bullet$} at 3 -1 \put{$\bullet$} at 3 2
\put{$\bullet$} at 4 -1 \put{$\bullet$} at 4 2
\put{$\bullet$} at 5 -1 \put{$\bullet$} at 5 2
\plot 3 -1 2 2 /
\plot 2 -1 3 2 /
\endpicture} = d_1 d_2
$$
is the diagrammatic representation of the matrix product,
$$
\left(\begin{array}{cccccccc}
0&0&0&0&0\\
0&0&0&1&0\\
1&0&0&0&0\\
0&1&0&0&0\\
0&0&1&0&0\\
\end{array}\right)
\left(\begin{array}{cccccccc}
0&1&0&0&0\\
0&0&0&0&0\\
0&0&0&0&0\\
0&0&1&0&0\\
0&0&0&1&0\\
\end{array}\right) 
=
\left(\begin{array}{cccccccc}
0&0&0&0&0\\
0&0&1&0&0\\
0&1&0&0&0\\
0&0&0&0&0\\
0&0&0&0&0\\
\end{array}\right).
$$

We say that an element of $R_n$ is {\it planar} if its diagram can be drawn (keeping inside of the rectangle formed by its vertices) without any edge crossings. We let $P_n\subseteq R_n$ denote the set of planar elements of $R_n$.  Our $R_6$ example is not planar. In the multiplication example above, the diagram $d_2$ is planar and  $d_1$ is not. Below are a few more examples of elements in $P_5$ (of rank 4, 2, 0, and 5, respectively). The fourth diagram is the identity in $P_5 \subseteq R_5$, and the third corresponds to the matrix of all 0s.
$$
{\beginpicture
\setcoordinatesystem units <0.3cm,0.15cm>         
\setplotarea x from 1 to 5, y from -1 to 2    
\linethickness=0.5pt                          
\put{$\bullet$} at 1 -1 \put{$\bullet$} at 1 2
\put{$\bullet$} at 2 -1 \put{$\bullet$} at 2 2
\put{$\bullet$} at 3 -1 \put{$\bullet$} at 3 2
\put{$\bullet$} at 4 -1 \put{$\bullet$} at 4 2
\put{$\bullet$} at 5 -1 \put{$\bullet$} at 5 2
\plot 1 -1  1 2 /
\plot 2 -1  3 2 /
\plot 3 -1 4 2  /
\plot 5 -1 5 2 /
\endpicture}\qquad
{\beginpicture
\setcoordinatesystem units <0.3cm,0.15cm>         
\setplotarea x from 1 to 5, y from -1 to 2    
\linethickness=0.5pt                          
\put{$\bullet$} at 1 -1 \put{$\bullet$} at 1 2
\put{$\bullet$} at 2 -1 \put{$\bullet$} at 2 2
\put{$\bullet$} at 3 -1 \put{$\bullet$} at 3 2
\put{$\bullet$} at 4 -1 \put{$\bullet$} at 4 2
\put{$\bullet$} at 5 -1 \put{$\bullet$} at 5 2
\plot 2 -1 1 2  /
\plot 3 -1 5 2  /
\endpicture} \qquad
{\beginpicture
\setcoordinatesystem units <0.3cm,0.15cm>         
\setplotarea x from 1 to 5, y from -1 to 2    
\linethickness=0.5pt                          
\put{$\bullet$} at 1 -1 \put{$\bullet$} at 1 2
\put{$\bullet$} at 2 -1 \put{$\bullet$} at 2 2
\put{$\bullet$} at 3 -1 \put{$\bullet$} at 3 2
\put{$\bullet$} at 4 -1 \put{$\bullet$} at 4 2
\put{$\bullet$} at 5 -1 \put{$\bullet$} at 5 2
\endpicture} \qquad
{\beginpicture
\setcoordinatesystem units <0.3cm,0.15cm>         
\setplotarea x from 1 to 5, y from -1 to 2    
\linethickness=0.5pt                          
\put{$\bullet$} at 1 -1 \put{$\bullet$} at 1 2
\put{$\bullet$} at 2 -1 \put{$\bullet$} at 2 2
\put{$\bullet$} at 3 -1 \put{$\bullet$} at 3 2
\put{$\bullet$} at 4 -1 \put{$\bullet$} at 4 2
\put{$\bullet$} at 5 -1 \put{$\bullet$} at 5 2
\plot 1 -1 1 2  /
\plot 2 -1 2 2  /
\plot 3 -1 3 2  /
\plot 4 -1 4 2  /
\plot 5 -1 5 2  /
\endpicture} = \id.
$$
It is easy to see (by drawing diagrams) that the product of two planar rook diagrams is again planar, so $P_n$ also forms a submonoid of $R_n$. The {\it only} invertible (rank $n$) planar rook diagram is the identity $\id$.

To construct a planar rook diagram of rank $k$, we choose $k$ vertices from each row. Then there is exactly one non-crossing way to connect them. Thus there are ${n \choose k}^2$ planar rook diagrams of rank $k$, and the number of planar rook diagrams is
\begin{equation*}
\vert P_n \vert = \sum_{k=0}^n {n \choose k}^2  = {2 n \choose n},
\end{equation*}
where we will let $P_0 = \{\emptyset\}$.
The last equality above is a well-known binomial identity. To see it in this setting,  choose any $n$ of the $2n$ vertices in the rook diagram. Let $k$ be the number of these chosen vertices that are in the top row (thus there are $n-k$ in the bottom row).  Connect (in {\it the} one and only non-crossing way) the $k$ chosen chosen vertices from the top row to the $k$ {\it not} chosen vertices from the bottom row.  

The algebraic properties of the rook monoid are studied in \cite{So}, \cite{Gr}, \cite{Re}, and \cite{Ha}.  The planar rook monoid appears in \cite{Re} as the order preserving ``partial permutations" of $\{1, 2, \ldots, n\}$,  and some combinatorics of $P_n$ arise in \cite{HL}.

\section{Planar Rook Diagrams Acting on Sets}

In this section, we construct $n \choose k$ dimensional vector spaces $V^n_k$, for  $0 \le k \le n$, and we will define a natural action (as linear transformations) of $P_n$ on these vector spaces. In this way, we can homomorphically represent multiplication in $P_n$ by multiplication of ${n \choose k}\times {n \choose k}$ matrices. Furthermore, we show that these matrix representations are all different, are irreducible, and include all the irreducible representations of $P_n$.

For a planar rook diagram $d$, let $\tau(d)$ and $\beta(d)$ denote the vertices in the top and bottom rows of $d$, respectively, that are incident  to an edge.  For example, 
$$
\hbox{if} \quad
d = {\beginpicture
\setcoordinatesystem units <0.25cm,0.12cm>         
\setplotarea x from 1 to 5, y from 0 to 3    
\linethickness=0.5pt                          
\put{$\bullet$} at 1 -1 \put{$\bullet$} at 1 3
\put{$\bullet$} at 2 -1 \put{$\bullet$} at 2 3
\put{$\bullet$} at 3 -1 \put{$\bullet$} at 3 3
\put{$\bullet$} at 4 -1 \put{$\bullet$} at 4 3
\put{$\bullet$} at 5 -1 \put{$\bullet$} at 5 3
\plot 1 -1 2 3  /
\plot 2 -1 3 3  /
\plot 5 -1 4 3  /
\endpicture} 
\quad \hbox{then}\quad
\beta(d) = \{1,2,5\}
\hbox{ and } \tau(d) = \{2,3,4\}.
$$
The sets $\tau(d)$ and $\beta(d)$ uniquely determine $d$ since there is only one planar way to connect the vertices by edges.  We can view $d$ as a 1-1 function with domain $\beta(d)$ and codomain $\tau(d).$   So, in our  example, $d(1) = 2, d(2)=3$, and $d(5)=4$.  

Now consider a subset $S = \{s_1, \ldots, s_k\}$ of order $k$ chosen from the set $\{1, 2, \ldots, n\}$.  If $d \in P_n$ and if $S$  is a subset of the domain $\beta(d)$ of $d$, then we can define an action of $d$ on $S$ by $d(S) = \{ d(s_1), d(s_2), \ldots, d(s_k)\}$.  Notice that $d(S)$ and $S$ have the same cardinality.

There are $2^n$ subsets of $\{1, 2, \ldots, n\}$, and we define a vector space $V^n$ over $\CC$ with dimension $2^n$ having a basis $\{\vec_S\}$ labeled by these subsets $S \subseteq \{1, \ldots, n\}$. Thus
\begin{equation}\label{setrep}
V^n = \CC\hbox{-span} \left\{ \ \vec_S \ \big\vert \  S \subseteq\{1, \ldots, n\} \ \right\}.
\end{equation}
We define an action of $P_n$ on $V^n$ as follows. 
For $d \in P_n$ and $S \subseteq \{1, \ldots, n\}$, define 
\begin{equation}
\label{actiononsets}
d\vec_S  = \begin{cases}
 \vec_{d(S)}, & \text{if $S \subseteq \beta(d)$, } \\
0, & \text{otherwise}. \\
\end{cases}
\end{equation}
This defines an action of $d$ on the basis of $V^n$ which we then extend linearly to all of $V^n$.
To illustrate with some examples, if we again let $
d = {\beginpicture
\setcoordinatesystem units <0.25cm,0.11cm>         
\setplotarea x from 1 to 5, y from 0 to 3    
\linethickness=0.5pt                          
\put{$\bullet$} at 1 -1 \put{$\bullet$} at 1 3
\put{$\bullet$} at 2 -1 \put{$\bullet$} at 2 3
\put{$\bullet$} at 3 -1 \put{$\bullet$} at 3 3
\put{$\bullet$} at 4 -1 \put{$\bullet$} at 4 3
\put{$\bullet$} at 5 -1 \put{$\bullet$} at 5 3
\plot 1 -1 2 3  /
\plot 2 -1 3 3  /
\plot 5 -1 4 3  /
\endpicture}$,
then $d\vec_{ \{  1,2,5 \} }  = \vec_{ \{  2, 3, 4 \} }$, $d \vec_{ \{  2, 5 \} }  = \vec_{ \{  3,4 \} }$, and $d \vec_{ \{  1,2,3 \} } = 0.$

It follows from diagram multiplication that $(d_1 d_2) \vec_S  = d_1 (d_2(\vec_S))$.   This means that $V^n$ is a ``module" for $P_n$. The map from $P_n$ to the set $\End(V^n)$ of linear transformations on $V^n$ is an injective monoid homomorphism.

For $0 \le k \le n$  consider the subspace of $V^n$ spanned by subsets of cardinality $k$,
\begin{equation}
\label{subsetrep}
V^n_k = \CC\hbox{-span} \left\{ \ \vec_S \ \big\vert \  S \subseteq\{1, \ldots, n\} \hbox{ and } \vert S \vert = k \ \right\}.
\end{equation}
Since the action of $P_n$ preserves the size of the subset (or sends it to the zero vector) we see that the $V^n_k$ are $P_n$-invariant submodules.
The following theorem describes the structure of $V^n$ as a module for $P_n$.

\begin{thm} 
\label{ModelDecomposition}  For all $n \ge 0$ and $0 \le k \le n$, we have
\begin{enumerate}
\item[(a)] $V^n_k$ is a $P_n$-module, and the $V^n_k$ are non-isomorphic for different $k$. 
\item[(b)] $V^n_k$ is irreducible (it contains no proper, nonzero $P_n$-invariant subspaces).  
\item[(c)]   $V^n$ decomposes as
$$
V^n \cong \bigoplus_{k=0}^n\ V^n_k,
$$
where each irreducible module appears with multiplicity 1.
\end{enumerate}
\end{thm}

\begin{proof} (a) The fact that $V^n_k$ is a module follows from the discussion preceding the Theorem. 
Since the dimensions of these modules are binomial coefficients, the only possible isomorphism could occur between $V^n_k$ and $V^n_{n-k}$. 
The element of the form $\pi_\ell = {\beginpicture
\setcoordinatesystem units <0.3cm,0.12cm>         
\setplotarea x from 1 to 5, y from 0 to 3    
\linethickness=0.5pt                          
\put{$\bullet$} at 1 -1 \put{$\bullet$} at 1 2
\put{$\bullet$} at 2 -1 \put{$\bullet$} at 2 2
\put{$\bullet$} at 3 -1 \put{$\bullet$} at 3 2
\put{$\bullet$} at 4 -1 \put{$\bullet$} at 4 2
\put{$\bullet$} at 5 -1 \put{$\bullet$} at 5 2
\plot 1 -1  1 2  /
\plot 2 -1  2 2 /
\plot 3 -1  3 2 /
\endpicture},$ which has $\ell$ vertical edges, acts by zero on $V^n_k$, with $k > \ell$. So the set of these $\pi_\ell, 0 \le \ell \le n$, will distinguish the $V^n_k$ from one another.

(b) To show that $V^n_k$ is irreducible, suppose that $W \subseteq V^n_k$ is a $P_n$-invariant subspace, and that $0 \not= w \in W$.  We expand $w$ in the basis as $w = \sum_{\vert S \vert = k} \lambda_S v_S$, with $\lambda_S \in \CC$.  Since $w \not=0$, there must be at least one $\lambda_S \not= 0$.  Let $d \in P_n$ be the unique planar diagram with $\tau(d) = \beta(d) = S$.  Then $d w  = \lambda_S v_S$, so $v_S \in W$.   Now let $S'$ be any other subset of order $k$ and let $d' \in P_n$ be the unique planar diagram with $\beta(d') = S$ and $\tau(d') = S'$.  Then, $d' \vec_S = \vec_{d'(S)} = \vec_{S'}$, so $\vec_{S'} \in W$. This shows that all the basis vectors of $V^n_k$ must be in $W$ and so $W = V^n_k$.

(c) The fact that $V^n$ decomposes as stated follows immediately from the fact that each $\vec_S$ appears in exactly one of the $V^n_k$.  In Section 3 we will prove that these are all  of the irreducible representations by showing that these are the only representations that show up in the  regular representation of $P_n$ acting on itself by multiplication.
\end{proof}

Let $P_0 = \{\emptyset\}$ and view $P_0 \subseteq P_1 \subseteq P_2 \subseteq \cdots$  by placing a vertical edge on the right of each diagram in $P_{n-1}$,  i.e., an edge that connects the $n$th vertex in each row.  For example,
$$
{\beginpicture
\setcoordinatesystem units <0.2cm,0.15cm>         
\setplotarea x from 1 to 5, y from 0 to 3    
\linethickness=0.5pt                          
\put{$\bullet$} at 1 -1 \put{$\bullet$} at 1 2
\put{$\bullet$} at 2 -1 \put{$\bullet$} at 2 2
\put{$\bullet$} at 3 -1 \put{$\bullet$} at 3 2
\put{$\bullet$} at 4 -1 \put{$\bullet$} at 4 2
\put{$\bullet$} at 5 -1 \put{$\bullet$} at 5 2
\plot 1 -1 2 2  /
\plot 2 -1 3 2  /
\plot 4 -1 4 2  /
\endpicture}  \rightarrow 
 {\beginpicture
\setcoordinatesystem units <0.2cm,0.15cm>         
\setplotarea x from 1 to 6, y from 0 to 3    
\linethickness=0.5pt                          
\put{$\bullet$} at 1 -1 \put{$\bullet$} at 1 2
\put{$\bullet$} at 2 -1 \put{$\bullet$} at 2 2
\put{$\bullet$} at 3 -1 \put{$\bullet$} at 3 2
\put{$\bullet$} at 4 -1 \put{$\bullet$} at 4 2
\put{$\bullet$} at 5 -1 \put{$\bullet$} at 5 2
\put{$\bullet$} at 6 -1 \put{$\bullet$} at 6 2
\plot 1 -1 2 2  /
\plot 2 -1 3 2  /
\plot 4 -1 4 2  /
\plot 6 -1 6 2 /
\endpicture}.
$$
It is natural to look at the restriction of the action of $P_n$ on $V^n_k$ to the submonoid $P_{n-1}$.  To this end, we construct the following subspaces of $V^n_k$,
\begin{align*}
V^n_{k,n} &= \CC\hbox{-span} \left\{  v_S\ \big\vert \ \vert S \vert = k, n \in S\right\}, \\
\widehat V^n_{k, n} &= \CC\hbox{-span} \left\{  v_S \ \big\vert \ \vert S \vert = k, n \not\in S\right\}.
\end{align*}
If $d \in P_{n-1}$, then by the way we embed $P_{n-1}$ into $P_n$, we have $n \in \tau(d)$ and $n \in \beta(d)$.  Thus it is always the case that for a subset $S \subseteq \{1, \ldots, n\}$, we have $n \in d(S)$ if and only if $n \in S$.   This means that under the the action defined in (\ref{actiononsets}), the subspaces $V^n_{k,n}$ and $\widehat V^n_{k, n}$ are  $P_{n-1}$-invariant.  

From the point of view of $P_{n-1}$, we see that
$$
V^n_{k,n} \cong V^{n-1}_{k-1} \quad\hbox{and}\quad
\widehat V^n_{k,  n} \cong V^{n-1}_{k},
$$
since, in the first case,  we are simply ignoring the element $n$, and in the second case, the basis vectors already are of the form $\vec_S$ with $S \subseteq \{1, \ldots, n-1\}$. Thus,  the vector space $V^n_k$ is irreducible under the $P_n$ action, but it breaks up into the following direct sum of irreducible $P_{n-1}$-invariant subspaces,
\begin{equation}\label{restrictionrule}
 V^n_k  \cong V^{n-1}_{k-1} \oplus V^{n-1}_k,
\end{equation}
where we drop the $V^{n-1}_{k-1}$ if $k=0$ and drop the $V^{n-1}_k$ if $k=n$.

\section{The Planar Rook Algebra}

In this section we let $P_n$ act on itself by multiplication; this is called the {\em regular representation} of $P_n$.

 The Artin-Wedderburn theory of semisimple algebras (see for example  \cite[Ch.\ IV]{CR} or \cite[Sec.\ 5]{HR}) states that if the regular representation of an algebra decomposes into a direct sum of irreducible modules, then (1) every irreducible module of the algebra is isomorphic to a summand of the regular representation and (2) every module of the algebra is isomorphic to a direct sum of irreducible modules. With this motivation, we construct an algebra associated to $P_n$, the planar rook algebra, and show explicitly that its regular representation reduces into a direct sum of modules each isomorphic to one of the $V^n_k$.

We define $\CC P_n$ to be the $\CC$-vector space with a  basis given   by the elements of $P_n$. That is,
$$
\CC P_n = \CC\hbox{-span} \left\{ \ d \ \big\vert \  d \in P_n \ \right\} = \left\{  
\sum_{d \in P_n}  \lambda_d  d\ \big\vert \lambda_d \in \CC \ \right\}.
$$
This is the vector space of all (formal) linear combinations of planar rook diagrams, and it has dimension
equal to the cardinality $\vert P_n\vert = {2n \choose n}$.  This complex vector space $\CC P_n$  is also equipped with a multiplication given by extending linearly the multiplication of diagrams in $P_n$.  This makes $\CC P_n$ an algebra over $\CC$ which we call the {\it planar rook algebra.}

It is interesting to notice that the diagram associated to the zero matrix,  {\beginpicture
\setcoordinatesystem units <0.2cm,0.12cm>         
\setplotarea x from 1 to 6, y from 0 to 3    
\linethickness=0.5pt                          
\put{$\bullet$} at 1 -1 \put{$\bullet$} at 1 2
\put{$\bullet$} at 2 -1 \put{$\bullet$} at 2 2
\put{$\bullet$} at 3 -1 \put{$\bullet$} at 3 2
\put{$\bullet$} at 4 -1 \put{$\bullet$} at 4 2
\put{$\bullet$} at 5 -1 \put{$\bullet$} at 5 2
\put{$\bullet$} at 6 -1 \put{$\bullet$} at 6 2
\endpicture}, is a basis element in this vector space, whereas the 0 vector is the linear combination with
all the $\lambda_d = 0$.

Since $\CC P_n$ is spanned by planar rook diagrams, an element  $d \in P_n$ acts naturally on the vector space $\CC P_n$ by multiplication on the left. That is, if $d \in P_n$ and $\vec = \sum_{b \in P_n} \lambda_b b \in \CC P_n$ we have 
$$
d \vec  = d \left(\sum_{b \in P_n} \lambda_b b\right)  =   \sum_{b \in P_n} \lambda_b d b.
$$

Multiplication of planar rook diagrams has the property that rank does not go up, i.e.,
$$
\rank(d_1 d_2) \le \min( \rank(d_1), \rank(d_2) ).
$$ 
Thus if we let $X^n_k$ be the span of the diagrams with rank less than or equal to $k$, we have a tower of $P_n$-invariant subspaces $X^n_0 \subseteq X^n_1 \subseteq \cdots \subseteq X^n_n$.   These $X^n_k$  are not irreducible and they do not decompose the space $\CC P_n$ into $P_n$-invariant subspaces. To accomplish such a decomposition, we first need to change to a different  but closely related basis.

 If $d_1, d_2 \in P_n$ we say that $d_1 \subseteq d_2$ if the edges of the diagram $d_1$ are a subset of the edges of the diagram of $d_2$.  If $d_1 \subseteq d_2$, we let $\vert d_2 \setminus d_1\vert = \rank(d_2) - \rank(d_1)$ or the number of edges in $d_2$ minus the number of edges in $d_1$.   Now define
\begin{equation}
\label{def:xbasis}
x_d = \sum_{d' \subseteq d}  (-1)^{\vert d \setminus d' \vert}  d'.
\end{equation}
For example,  if
$
d = {\beginpicture
\setcoordinatesystem units <0.2cm,0.15cm>         
\setplotarea x from 1 to 5, y from 0 to 3    
\linethickness=0.5pt                          
\put{$\bullet$} at 1 -1 \put{$\bullet$} at 1 2
\put{$\bullet$} at 2 -1 \put{$\bullet$} at 2 2
\put{$\bullet$} at 3 -1 \put{$\bullet$} at 3 2
\put{$\bullet$} at 4 -1 \put{$\bullet$} at 4 2
\put{$\bullet$} at 5 -1 \put{$\bullet$} at 5 2
\plot 1 -1 2 2  /
\plot 2 -1 3 2  /
\plot 4 -1 4 2  /
\endpicture}
$
then
$$
x_{d} = 
 {\beginpicture
\setcoordinatesystem units <0.2cm,0.15cm>         
\setplotarea x from 1 to 5, y from 0 to 3    
\linethickness=0.5pt                          
\put{$\bullet$} at 1 -1 \put{$\bullet$} at 1 2
\put{$\bullet$} at 2 -1 \put{$\bullet$} at 2 2
\put{$\bullet$} at 3 -1 \put{$\bullet$} at 3 2
\put{$\bullet$} at 4 -1 \put{$\bullet$} at 4 2
\put{$\bullet$} at 5 -1 \put{$\bullet$} at 5 2
\plot 1 -1 2 2  /
\plot 2 -1 3 2  /
\plot 4 -1 4 2  /
\endpicture}
-
 {\beginpicture
\setcoordinatesystem units <0.2cm,0.15cm>         
\setplotarea x from 1 to 5, y from 0 to 3    
\linethickness=0.5pt                          
\put{$\bullet$} at 1 -1 \put{$\bullet$} at 1 2
\put{$\bullet$} at 2 -1 \put{$\bullet$} at 2 2
\put{$\bullet$} at 3 -1 \put{$\bullet$} at 3 2
\put{$\bullet$} at 4 -1 \put{$\bullet$} at 4 2
\put{$\bullet$} at 5 -1 \put{$\bullet$} at 5 2
\plot 2 -1 3 2  /
\plot 4 -1 4 2  /
\endpicture}
- {\beginpicture
\setcoordinatesystem units <0.2cm,0.15cm>         
\setplotarea x from 1 to 5, y from 0 to 3    
\linethickness=0.5pt                          
\put{$\bullet$} at 1 -1 \put{$\bullet$} at 1 2
\put{$\bullet$} at 2 -1 \put{$\bullet$} at 2 2
\put{$\bullet$} at 3 -1 \put{$\bullet$} at 3 2
\put{$\bullet$} at 4 -1 \put{$\bullet$} at 4 2
\put{$\bullet$} at 5 -1 \put{$\bullet$} at 5 2
\plot 1 -1 2 2  /
\plot 4 -1 4 2  /
\endpicture}
- {\beginpicture
\setcoordinatesystem units <0.2cm,0.15cm>         
\setplotarea x from 1 to 5, y from 0 to 3    
\linethickness=0.5pt                          
\put{$\bullet$} at 1 -1 \put{$\bullet$} at 1 2
\put{$\bullet$} at 2 -1 \put{$\bullet$} at 2 2
\put{$\bullet$} at 3 -1 \put{$\bullet$} at 3 2
\put{$\bullet$} at 4 -1 \put{$\bullet$} at 4 2
\put{$\bullet$} at 5 -1 \put{$\bullet$} at 5 2
\plot 1 -1 2 2  /
\plot 2 -1 3 2  /
\endpicture}
+
 {\beginpicture
\setcoordinatesystem units <0.2cm,0.15cm>         
\setplotarea x from 1 to 5, y from 0 to 3    
\linethickness=0.5pt                          
\put{$\bullet$} at 1 -1 \put{$\bullet$} at 1 2
\put{$\bullet$} at 2 -1 \put{$\bullet$} at 2 2
\put{$\bullet$} at 3 -1 \put{$\bullet$} at 3 2
\put{$\bullet$} at 4 -1 \put{$\bullet$} at 4 2
\put{$\bullet$} at 5 -1 \put{$\bullet$} at 5 2
\plot 4 -1 4 2  /
\endpicture}
+
 {\beginpicture
\setcoordinatesystem units <0.2cm,0.15cm>         
\setplotarea x from 1 to 5, y from 0 to 3    
\linethickness=0.5pt                          
\put{$\bullet$} at 1 -1 \put{$\bullet$} at 1 2
\put{$\bullet$} at 2 -1 \put{$\bullet$} at 2 2
\put{$\bullet$} at 3 -1 \put{$\bullet$} at 3 2
\put{$\bullet$} at 4 -1 \put{$\bullet$} at 4 2
\put{$\bullet$} at 5 -1 \put{$\bullet$} at 5 2
\plot 2 -1 3 2  /
\endpicture}
+
 {\beginpicture
\setcoordinatesystem units <0.2cm,0.15cm>         
\setplotarea x from 1 to 5, y from 0 to 3    
\linethickness=0.5pt                          
\put{$\bullet$} at 1 -1 \put{$\bullet$} at 1 2
\put{$\bullet$} at 2 -1 \put{$\bullet$} at 2 2
\put{$\bullet$} at 3 -1 \put{$\bullet$} at 3 2
\put{$\bullet$} at 4 -1 \put{$\bullet$} at 4 2
\put{$\bullet$} at 5 -1 \put{$\bullet$} at 5 2
\plot 1 -1 2 2  /
\endpicture}
-
 {\beginpicture
\setcoordinatesystem units <0.2cm,0.15cm>         
\setplotarea x from 1 to 5, y from 0 to 3    
\linethickness=0.5pt                          
\put{$\bullet$} at 1 -1 \put{$\bullet$} at 1 2
\put{$\bullet$} at 2 -1 \put{$\bullet$} at 2 2
\put{$\bullet$} at 3 -1 \put{$\bullet$} at 3 2
\put{$\bullet$} at 4 -1 \put{$\bullet$} at 4 2
\put{$\bullet$} at 5 -1 \put{$\bullet$} at 5 2
\endpicture}.
$$
Under any ordering on the planar rook diagrams that extends the partial ordering given by rank (i.e., $a$ comes before $b$ if $\rank(a) < \rank(b)$), the transition matrix from the basis $\{ d\ \vert\  d\in P_n\}$ to the set $\{x_d \ \vert \ d \in P_n\}$  is upper triangular with 1s on the diagonal. Thus  $\{x_d \ \vert \ d \in P_n\}$ is also a basis for $\CC P_n$.

This  change of basis was necessary to make the second case in  the statement in the next proposition come out to 0.  Notice how close this statement is to \eqref{actiononsets}.  We are now realizing the subset action inside of $\CC P_n$.

\begin{prop} 
\label{prop:inclusionexclusion} Let $a, d \in P_n$. Then
\begin{equation}\label{actiononx}
d x_{a}  =
 \begin{cases}
x_{da}, & \hbox{if $\tau(a) \subseteq \beta(d)$}  \\
0, &             \hbox{otherwise. } \\
\end{cases}
\end{equation}
\end{prop}

\begin{proof}  If $\tau(a) \subseteq \beta(d)$ then multiplication on the left (or top) by $d$ on any $d' \subseteq a$ simply rearranges the top vertices of $d'$ to their corresponding position in $da$, and the result follows by the definition of $x_{da}$.

If $\tau(a) \not\subseteq \beta(d)$, then let $i \in \tau(a)$ such that $i \not \in \beta(d)$.  Consider the diagram $p_i$, which is the same as the identity element $\id$ except that the edge connecting the $i$th vertex in each row is removed.  For example, in $P_5$, we have
$p_4 =  {\beginpicture
\setcoordinatesystem units <0.3cm,0.12cm>         
\setplotarea x from 1 to 5, y from 0 to 3    
\linethickness=0.5pt                          
\put{$\bullet$} at 1 -1 \put{$\bullet$} at 1 2
\put{$\bullet$} at 2 -1 \put{$\bullet$} at 2 2
\put{$\bullet$} at 3 -1 \put{$\bullet$} at 3 2
\put{$\bullet$} at 4 -1 \put{$\bullet$} at 4 2
\put{$\bullet$} at 5 -1 \put{$\bullet$} at 5 2
\plot 1 -1  1 2  /
\plot 2 -1  2 2 /
\plot 3 -1  3 2 /
\plot 5 -1  5 2 /
\endpicture}.
$
 Then  $d p_i  = d$, since $i \not\in \beta(d)$, and
$$
p_i x_a = \sum_{d' \subseteq a} (-1)^{\vert a\setminus d'\vert} p_i d'  
= \sum_{d' \subseteq a  \atop i \in \tau(d')} (-1)^{\vert a \setminus d' \vert} p_i d'  +  \sum_{d' \subseteq a \atop i \not\in \tau(d')} (-1)^{\vert a \setminus d' \vert} p_i d' .
$$
Now, if $i \not \in \tau(d')$ then $p_i d'  = d'$, and if $i \in \tau(d')$ then $p_i d' $ is the same diagram as $d'$ except with the edge connected to the $i$th vertex (in the top row of $d'$) removed.  There is a bijection between $\{d' \subseteq d\ \vert \  i \in \tau(d') \}$ and $\{d' \subseteq d\  \vert i \not\in \tau(d') \}$ given by removing the edge connected to the $i$th vertex (equivalently, multiplying by $p_i$). This bijection changes the sign $(-1)^{\vert a \setminus d' \vert}$, so the two summations displayed above cancel one another giving $p_i x_a= 0$. Thus $d x_a  = d p_i x_a =  0$.
\end{proof}

For diagrams $a$ and $d$ we have $\rank(a) = \rank(da)$ if and only if $\tau(a) \subseteq \beta(d)$. Thus from  (\ref{actiononx}), we see that $d x_a = 0$ unless $\rank(a)=\rank(d a)$. It follows that the subspace
$$
W^{n,k} = \CC\hbox{-span} \left\{\  x_a \  \big\vert \ \rank(a) = k \right\},
$$
is a $P_n$-invariant subspace of $\CC P_n$.   Notice that the action of $d$ on $x_a$ in  (\ref{actiononx})  does not change the bottom row $\beta(a)$. That is, $\beta(a) = \beta(da)$ when $\tau(a) \subseteq \beta(d)$.  Thus, if we let
$$
W^{n,k}_T = \CC\hbox{-span} \left\{\  x_a \  \big\vert \ \rank(a) = k,  \beta(a) = T\ \right\},
$$
then for each such $T$, we have that $W^{n,k}_T$ is a $P_n$-invariant subspace of $W^{n,k}$ and for any subset $U$ with $\vert U \vert = \vert T \vert = k$, we have
\begin{equation}\label{twoisomorphisms}
W^{n,k}_T \cong W^{n,k}_U \cong  V^n_k \qquad\hbox{ as $P_n$-invariant subspaces of $\CC P_n$}.
\end{equation}
The last isomorphism comes from the fact that the action of $d \in P_n$ on $x_a$ in (\ref{actiononx}) is the same as the action of $d$ on $v_S$ in (\ref{actiononsets}), where $S = \tau(a)$.  

For subsets $S, T$ of $\{1, \ldots, n\}$ with $\vert S \vert = \vert T\vert$, we define
$$
x_{{}_{S,T}} = x_d, \qquad
\begin{array}{l}
 \hbox{where $d$ is the unique planar rook diagram } \\
\hbox{with $\tau(d) = S$ and $\beta(d) = T$}.\\
\end{array}
$$
For example, the diagram in the example after equation \eqref{def:xbasis} is denoted $x_{\{2,3,4\},\{1,2,4\}}$. In this notation, the isomorphism in (\ref{twoisomorphisms}) is given explicitly on basis elements by $x_{{}_{S,T}} \leftrightarrow x_{{}_{S,U}} \leftrightarrow v_S$. 

Inside of $W^{n,k}$ we have found ${n \choose k}$ copies of the $P_n$-invariant subspaces $W^{n,k}_T$ (one for each choice of $T$), and each of these is  isomorphic to $V^n_k$. Thus we have explicitly constructed the decompositions in part (a) of the following theorem.  Part (b) follows from the fact that every irreducible module must appear as a component in the regular representation.

\begin{thm}  \label{thm:regularrepresentation}
\begin{enumerate}
\item[(a)] The decomposition of $\CC P_n$ into $P_n$-invariant subspaces is given by
$$
\CC P_n =  \bigoplus_{k=0}^n W^{n,k}
= \bigoplus_{k=0}^n \bigoplus_{\vert T\vert = k} W^{n,k}_T  \cong  \bigoplus_{k=0}^n  {n \choose k} V^n_k.
$$
\item[(b)]  The set $\{ V_k^n \ \vert\ 0 \le k \le n\  \}$ is a complete set of irreducible  $\CC P_n$-modules.
\end{enumerate}
\end{thm} 

In the previous theorem, the modules $W^{n,k}$ are the ``isotypic components"  which consist of a
sum of all of the irreducible subspaces that are isomorphic to $V^n_k$.  Notice also that the dimension and the multiplicity of $V^n_k$ in $\CC P_n$ is $n \choose k$.  Finally, since the irreducible modules that appear in $\CC P_n$ are exactly the $V^n_k$, we know that these form a complete set of irreducible modules as claimed in Theorem \ref{ModelDecomposition}.

\begin{prop}\label{matrixunits}  For subsets $S,T,U,V$ of $\{1, \ldots, n\}$ with $\vert S\vert  = \vert T\vert $ and $\vert U\vert =\vert V\vert $ we have
$$
x_{{}_{S,T}}  x_{{}_{U,V}} = \begin{cases}
x_{{}_{S,V}}, & \text{if $T=U$},\\
0, & \text{if $T\not=U$}.
\end{cases}
$$
\end{prop}

\begin{proof} If $T \not= U$, then there exists $i \in T$ with $i \not\in U$ or $i \in U$ with $i \not\in T$.  The same argument as in Proposition \ref{prop:inclusionexclusion} shows that $x_{{}_{S,T}} x_{{}_{U,V}} = 0$.  If $T = U$ then let $a,b \in P_n$ be the diagrams such that $x_{{}_{S,T}} = x_a$ and $x_{{}_{U,V}} = x_b$.  By \eqref{actiononx} we have that $a x_b = x_{ab}$ and $a' x_b = 0$ for every $a' \subseteq a$ with $a' \not= a$.  So by the definition of $x_a$ we see that $x_a x_b = x_{ab}$. 
\end{proof}

Proposition \ref{matrixunits}  tells us that the $x_{{}_{S,T}}$ behave just like the matrix units $E_{i,j}$ which have a 1 in row $i$ and column $j$ and 0 everywhere else.  This correspondence reveals the structure of the planar rook algebra, given in the following corollary.

\begin{cor}\label{matrixalgebra} 
$$
\CC P_n \cong \bigoplus_{k=0}^n {\bf Mat}\!\left( {n \choose k}, {n \choose k} \right),
$$
where ${\bf Mat}\!\left( m, m \right)$ is the algebra of all $m \times m$ complex matrices.
\end{cor}

\section{The Bratteli Diagram is Pascal's Triangle}

The binomial coefficients have appeared in a very natural way throughout the representation theory of $P_n$.  For example,  by comparing dimensions on both sides of the decomposition of $V^n$ in 
 Theorem \ref{ModelDecomposition}, we get 
\begin{equation}\label{BinomialIdentity1}
2^n = \sum_{k=0}^n {n \choose k}.
\end{equation}
By computing dimensions on both sides of the decomposition of $\CC P_n$ in Theorem  \ref{thm:regularrepresentation}, we have
\begin{equation}\label{BinomialIdentity2}
{2 n \choose n} = \sum_{k=0}^n {n \choose k}^2.
\end{equation}
By comparing dimensions of the decomposition of $V^n_k$ into irreducible modules for $P_{n-1}$ in 
(\ref{restrictionrule}), we get
\begin{equation}
{n \choose k} = {n-1 \choose k-1} + {n-1 \choose k}.
\end{equation}
These are well-known binomial identities.  For a beautiful discussion of combinatorial proofs of binomial identities such as this see \cite{BQ}. We can view the work in this article as representation-theoretic  interpretations of these binomial identities.

Pascal's triangle itself arises naturally through the representation theory of $P_n$. 
The {\it Bratteli diagram} (see for example \cite{GHJ})  for the tower $P_0 \subseteq P_1 \subseteq P_2 \subseteq \cdots $ is the infinite rooted tree whose vertices are the irreducible representations $V^n_k$ and whose edges correspond to the restriction rules from $P_n$ to $P_{n-1}$.  Specifically there is an edge from $V^n_k$ to $V^{n-1}_\ell$ if and only if $V^{n-1}_\ell$ appears as a summand when $V^n_k$ is viewed as a module for $P_{n-1}$.  According to the rules in \eqref{restrictionrule} we get the Bratteli diagram shown in Figure \ref{fig:Pascal}.  The dimensions of these modules gives Pascal's triangle.

\begin{figure}[h] 
   \centering
   \includegraphics[width=3.5in]{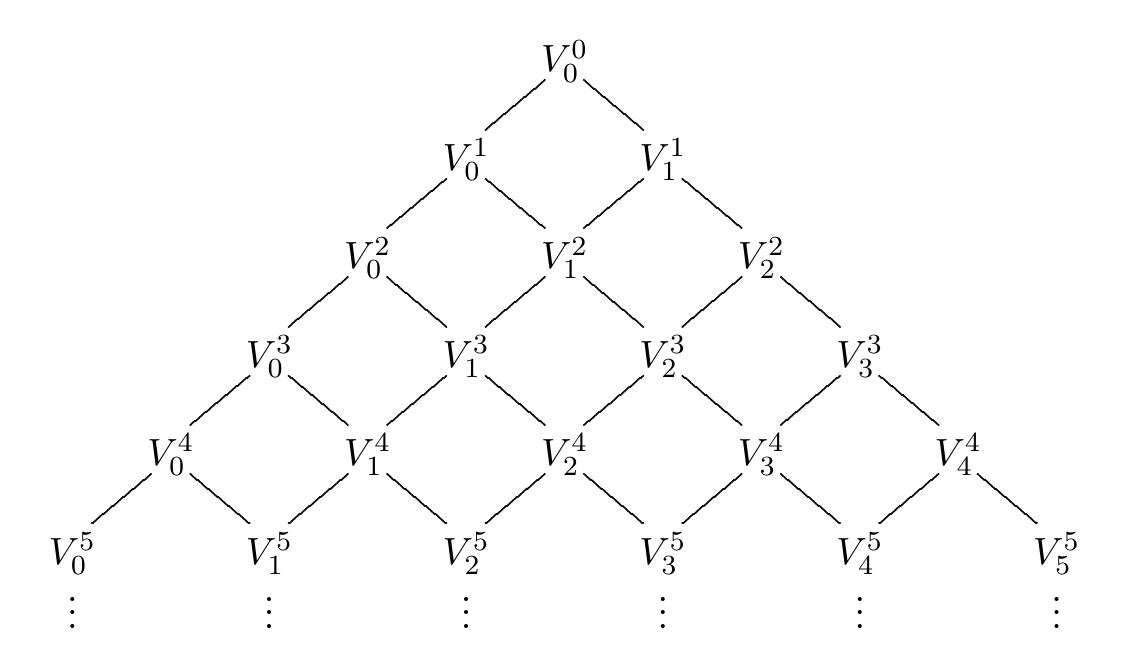} 
   \caption{The Bratteli diagram for the tower of containments  of planar rook monoids $P_0 \subseteq P_1 \subseteq P_2 \subseteq \cdots$.  Counting dimensions gives Pascal's triangle.}
   \label{fig:Pascal}
\end{figure}

\section{The Character Table is Pascal's Triangle}

For each irreducible representation $V^n_k$,  its character $\chi^n_k$ is the $\CC$-valued function that gives the trace of the 
$d \in P_n$ as a linear transformation on $V^n_k$. In this section, we show that the table of character values for $P_n$ is given by the first $n$ rows of Pascal's triangle.  The characters are linearly independent functions (over $\CC$), so from the character of any finite representation you can identify the isomorphism class of the representation.

For a planar rook diagram $d \in P_n$ we say that an edge in $d$ is {\it vertical} if it connects the $i$th vertex in the top row to the $i$th vertex in the bottom row for some $1 \le i \le n$.  We also say that a vertex that is not incident to an edge in $d$ is an {\it isolated} vertex.   Suppose that $d$ is a diagram such that its $i$th vertex in the top row is isolated.  As in Section 3, let $p_i$ be the diagram with obtained from the identity $\id$ by deleting the $i$th edge.  Then $p_i d = d$ and $d' = d p_i$ has the property that the $i$th vertex in both the top and bottom row is isolated.  For example,

$$
d =    {\beginpicture
\setcoordinatesystem units <0.2cm,0.15cm>         
\setplotarea x from 1 to 5, y from 0 to 3    
\linethickness=0.5pt                          
\put{$\bullet$} at 1 -1 \put{$\bullet$} at 1 2
\put{$\bullet$} at 2 -1 \put{$\bullet$} at 2 2
\put{$\bullet$} at 3 -1 \put{$\bullet$} at 3 2
\put{$\bullet$} at 4 -1 \put{$\bullet$} at 4 2
\put{$\bullet$} at 5 -1 \put{$\bullet$} at 5 2
\plot 1 -1 2 2  /
\plot 3 -1 4 2 /
\plot 5 -1 5 2 /
\endpicture} =
\begin{array}{l}
{\beginpicture
\setcoordinatesystem units <0.2cm,0.15cm>         
\setplotarea x from 1 to 5, y from 0 to 3    
\linethickness=0.5pt                          
\put{$\bullet$} at 1 -1 \put{$\bullet$} at 1 2
\put{$\bullet$} at 2 -1 \put{$\bullet$} at 2 2
\put{$\bullet$} at 3 -1 \put{$\bullet$} at 3 2
\put{$\bullet$} at 4 -1 \put{$\bullet$} at 4 2
\put{$\bullet$} at 5 -1 \put{$\bullet$} at 5 2
\plot 1 -1 1 2  /
\plot 2 -1 2 2  /
\plot 4 -1 4 2  /
\plot 5 -1 5 2  /
\endpicture}  = p_3\\
{\beginpicture
\setcoordinatesystem units <0.2cm,0.15cm>         
\setplotarea x from 1 to 5, y from 0 to 3    
\linethickness=0.5pt                          
\put{$\bullet$} at 1 -1 \put{$\bullet$} at 1 2
\put{$\bullet$} at 2 -1 \put{$\bullet$} at 2 2
\put{$\bullet$} at 3 -1 \put{$\bullet$} at 3 2
\put{$\bullet$} at 4 -1 \put{$\bullet$} at 4 2
\put{$\bullet$} at 5 -1 \put{$\bullet$} at 5 2
\plot 1 -1 2 2  /
\plot 3 -1 4 2 /
\plot 5 -1 5 2 /
\endpicture} = d
\end{array}
\qquad\hbox{and}\qquad
\begin{array}{r}
 d = {\beginpicture
\setcoordinatesystem units <0.2cm,0.15cm>         
\setplotarea x from 1 to 5, y from 0 to 3    
\linethickness=0.5pt                          
\put{$\bullet$} at 1 -1 \put{$\bullet$} at 1 2
\put{$\bullet$} at 2 -1 \put{$\bullet$} at 2 2
\put{$\bullet$} at 3 -1 \put{$\bullet$} at 3 2
\put{$\bullet$} at 4 -1 \put{$\bullet$} at 4 2
\put{$\bullet$} at 5 -1 \put{$\bullet$} at 5 2
\plot 1 -1 2 2  /
\plot 3 -1 4 2 /
\plot 5 -1 5 2 /
\endpicture}
 \\
p_3 = {\beginpicture
\setcoordinatesystem units <0.2cm,0.15cm>         
\setplotarea x from 1 to 5, y from 0 to 3    
\linethickness=0.5pt                          
\put{$\bullet$} at 1 -1 \put{$\bullet$} at 1 2
\put{$\bullet$} at 2 -1 \put{$\bullet$} at 2 2
\put{$\bullet$} at 3 -1 \put{$\bullet$} at 3 2
\put{$\bullet$} at 4 -1 \put{$\bullet$} at 4 2
\put{$\bullet$} at 5 -1 \put{$\bullet$} at 5 2
\plot 1 -1 1 2  /
\plot 2 -1 2 2  /
\plot 4 -1 4 2  /
\plot 5 -1 5 2  /
\endpicture}  \\ \end{array}
=   {\beginpicture
\setcoordinatesystem units <0.2cm,0.15cm>         
\setplotarea x from 1 to 5, y from 0 to 3    
\linethickness=0.5pt                          
\put{$\bullet$} at 1 -1 \put{$\bullet$} at 1 2
\put{$\bullet$} at 2 -1 \put{$\bullet$} at 2 2
\put{$\bullet$} at 3 -1 \put{$\bullet$} at 3 2
\put{$\bullet$} at 4 -1 \put{$\bullet$} at 4 2
\put{$\bullet$} at 5 -1 \put{$\bullet$} at 5 2
\plot 1 -1 2 2  /
\plot 5 -1 5 2 /
\endpicture} = d'
$$

\def\Tr{{\rm Tr}}
Now, the point of all this is that for any matrix trace $\Tr$ we have $\Tr(ab) = \Tr(ba)$, so in this case $\Tr(d) = \Tr(p_3 d) = \Tr(d p_3) = \Tr(d')$.  So by iterating this process, we see that for any matrix trace we have $\Tr(d) = \Tr(d')$ where $d'$ is the diagram $d$ with all of its non-vertical edges removed.

Furthermore, we can use the following trick to move all of the vertical edges to the left of the diagram.  In the picture below, we see that $d = RLd$ and $d' = LdR$ has the vertical edge moved one position to the left. Furthermore $\Tr(d) = \Tr(RLd) = \Tr(LdR) = \Tr(d')$.
$$
d =    {\beginpicture
\setcoordinatesystem units <0.2cm,0.15cm>         
\setplotarea x from 1 to 5, y from 0 to 3    
\linethickness=0.5pt                          
\put{$\bullet$} at 1 -1 \put{$\bullet$} at 1 2
\put{$\bullet$} at 2 -1 \put{$\bullet$} at 2 2
\put{$\bullet$} at 3 -1 \put{$\bullet$} at 3 2
\put{$\bullet$} at 4 -1 \put{$\bullet$} at 4 2
\put{$\bullet$} at 5 -1 \put{$\bullet$} at 5 2
\plot 1 -1 1 2  /
\plot 2 -1 2 2  /
\plot 4 -1 4 2  /
\plot 5 -1 5 2  /
\endpicture} =
\begin{array}{l}
{\beginpicture
\setcoordinatesystem units <0.2cm,0.15cm>         
\setplotarea x from 1 to 5, y from 0 to 3    
\linethickness=0.5pt                          
\put{$\bullet$} at 1 -1 \put{$\bullet$} at 1 2
\put{$\bullet$} at 2 -1 \put{$\bullet$} at 2 2
\put{$\bullet$} at 3 -1 \put{$\bullet$} at 3 2
\put{$\bullet$} at 4 -1 \put{$\bullet$} at 4 2
\put{$\bullet$} at 5 -1 \put{$\bullet$} at 5 2
\plot 1 -1 1 2  /
\plot 2 -1 2 2  /
\plot 3 -1 4 2  /
\plot 5 -1 5 2  /
\endpicture}  = R\\
{\beginpicture
\setcoordinatesystem units <0.2cm,0.15cm>         
\setplotarea x from 1 to 5, y from 0 to 3    
\linethickness=0.5pt                          
\put{$\bullet$} at 1 -1 \put{$\bullet$} at 1 2
\put{$\bullet$} at 2 -1 \put{$\bullet$} at 2 2
\put{$\bullet$} at 3 -1 \put{$\bullet$} at 3 2
\put{$\bullet$} at 4 -1 \put{$\bullet$} at 4 2
\put{$\bullet$} at 5 -1 \put{$\bullet$} at 5 2
\plot 1 -1 1 2  /
\plot 2 -1 2 2  /
\plot 3 2 4  -1  /
\plot 5 -1 5 2  /
\endpicture}  = L\\
{\beginpicture
\setcoordinatesystem units <0.2cm,0.15cm>         
\setplotarea x from 1 to 5, y from 0 to 3    
\linethickness=0.5pt                          
\put{$\bullet$} at 1 -1 \put{$\bullet$} at 1 2
\put{$\bullet$} at 2 -1 \put{$\bullet$} at 2 2
\put{$\bullet$} at 3 -1 \put{$\bullet$} at 3 2
\put{$\bullet$} at 4 -1 \put{$\bullet$} at 4 2
\put{$\bullet$} at 5 -1 \put{$\bullet$} at 5 2
\plot 1 -1 1 2  /
\plot 2 -1 2 2  /
\plot 4 -1 4 2  /
\plot 5 -1 5 2  /
\endpicture} = d
\end{array}
\qquad\hbox{and}\qquad
 \begin{array}{r}
L = {\beginpicture
\setcoordinatesystem units <0.2cm,0.15cm>         
\setplotarea x from 1 to 5, y from 0 to 3    
\linethickness=0.5pt                          
\put{$\bullet$} at 1 -1 \put{$\bullet$} at 1 2
\put{$\bullet$} at 2 -1 \put{$\bullet$} at 2 2
\put{$\bullet$} at 3 -1 \put{$\bullet$} at 3 2
\put{$\bullet$} at 4 -1 \put{$\bullet$} at 4 2
\put{$\bullet$} at 5 -1 \put{$\bullet$} at 5 2
\plot 1 -1 1 2  /
\plot 2 -1 2 2  /
\plot 3 2 4  -1  /
\plot 5 -1 5 2  /
\endpicture} \\
d= {\beginpicture
\setcoordinatesystem units <0.2cm,0.15cm>         
\setplotarea x from 1 to 5, y from 0 to 3    
\linethickness=0.5pt                          
\put{$\bullet$} at 1 -1 \put{$\bullet$} at 1 2
\put{$\bullet$} at 2 -1 \put{$\bullet$} at 2 2
\put{$\bullet$} at 3 -1 \put{$\bullet$} at 3 2
\put{$\bullet$} at 4 -1 \put{$\bullet$} at 4 2
\put{$\bullet$} at 5 -1 \put{$\bullet$} at 5 2
\plot 1 -1 1 2  /
\plot 2 -1 2 2  /
\plot 4 -1 4 2  /
\plot 5 -1 5 2  /
\endpicture}  \\
R={\beginpicture
\setcoordinatesystem units <0.2cm,0.15cm>         
\setplotarea x from 1 to 5, y from 0 to 3    
\linethickness=0.5pt                          
\put{$\bullet$} at 1 -1 \put{$\bullet$} at 1 2
\put{$\bullet$} at 2 -1 \put{$\bullet$} at 2 2
\put{$\bullet$} at 3 -1 \put{$\bullet$} at 3 2
\put{$\bullet$} at 4 -1 \put{$\bullet$} at 4 2
\put{$\bullet$} at 5 -1 \put{$\bullet$} at 5 2
\plot 1 -1 1 2  /
\plot 2 -1 2 2  /
\plot 3 -1 4 2  /
\plot 5 -1 5 2  /
\endpicture}  \\
\end{array}
=   {\beginpicture
\setcoordinatesystem units <0.2cm,0.15cm>         
\setplotarea x from 1 to 5, y from 0 to 3    
\linethickness=0.5pt                          
\put{$\bullet$} at 1 -1 \put{$\bullet$} at 1 2
\put{$\bullet$} at 2 -1 \put{$\bullet$} at 2 2
\put{$\bullet$} at 3 -1 \put{$\bullet$} at 3 2
\put{$\bullet$} at 4 -1 \put{$\bullet$} at 4 2
\put{$\bullet$} at 5 -1 \put{$\bullet$} at 5 2
\plot 1 -1 1 2  /
\plot 2 -1 2 2  /
\plot 3 -1 3 2  /
\plot 5 -1 5 2  /
\endpicture}  = d'
$$

By iterating this process, we see that the character value of $d$ will be the same as the character value on one of the diagrams,
$$
\pi_\ell = \underbrace{{\beginpicture
\setcoordinatesystem units <0.4cm,0.2cm>         
\setplotarea x from 1 to 5.5, y from 0 to 3    
\linethickness=0.5pt                          
\put{$\bullet$} at 1 -1 \put{$\bullet$} at 1 2
\put{$\bullet$} at 2 -1 \put{$\bullet$} at 2 2
\put{$\bullet$} at 3 -1 \put{$\bullet$} at 3 2
\put{$\cdots$} at 4 -1 \put{$\cdots$} at 4 2
\put{$\bullet$} at 5 -1 \put{$\bullet$} at 5 2
\plot 1 -1 1 2  /
\plot 2 -1 2 2  /
\plot 3 -1 3 2  /
\plot 5 -1 5 2  /
\endpicture}}_\ell
{\beginpicture
\setcoordinatesystem units <0.4cm,0.2cm>         
\setplotarea x from 1 to 5, y from 0 to 3    
\linethickness=0.5pt                          
\put{$\bullet$} at 1 -1 \put{$\bullet$} at 1 2
\put{$\bullet$} at 2 -1 \put{$\bullet$} at 2 2
\put{$\bullet$} at 3 -1 \put{$\bullet$} at 3 2
\put{$\cdots$} at 4 -1 \put{$\cdots$} at 4 2
\put{$\bullet$} at 5 -1 \put{$\bullet$} at 5 2
\endpicture}, \qquad 0 \le \ell \le n,
$$
where $\ell$ is the number of vertical edges in $d$. The set of diagrams $\{ \pi_\ell\ \vert \ 0 \le i \le n \}$ is analogous to a set of conjugacy class representatives in a group in the sense that any trace is completely determined by its value on one of these diagrams.

In the next theorem, we show that the trace of $d \in P_n$ on the representation $V^n_k$ is a binomial coefficient.  The proof is to count subsets  fixed by $d$.

\begin{thm} \label{thm:chars} For $0 \le k \le n$ and $d \in P_n$, the value of the irreducible character is given by
$$
\chi^n_k(d) = 
\begin{cases}
{\ell \choose k}, & \text{ if $k \le \ell$}, \\
0, & \text{ if $k > \ell$}, \\
\end{cases}
$$
where $\ell$ is the number of vertical edges in $d$.
\end{thm}

\begin{proof}  The elements $d \in P_n$ permute (or send to 0) the vectors $v_S$ which span $V^n_k$.  The $v_S$-$v_S$ entry of the matrix of $d$ will be 1 if $d(S) = S$ and 0 otherwise.  This tells us that the character $\chi^n_k(d)$ gives the number of fixed points of $d$.  By our discussion above it suffices to let $d = \pi_\ell$.  Now, $\pi_\ell$ will fix $S$ if and only if $S \subseteq \{1,  \ldots, \ell \}$.  And for $v_S$ to be a basis element of $V^n_k$ it must be a subset of $\{1,  \ldots, n\}$ with cardinality $k$. Thus, the trace is the number of subsets of $\{1,  \ldots, \ell \}$ of cardinality $k$, or $\ell \choose k$, as desired.
\end{proof}

\section{Further Thoughts}

\noindent 
Here are a few more observations that make fun exercises.  

\begin{enumerate}
\item Let $\psi^n(d)$ denote the trace of $d \in P_n$ on the regular representation $\CC P_n$.  Then by a counting argument $\psi^n(d) = {n+\ell \choose \ell}$, where $\ell$ is the number of vertical edges in $d$.  Using the decomposition of the regular representation into irreducibles we arrive at the binomial identity
$$
\sum_{k=0}^n {n \choose k} \chi^n_k(d) = \sum_{k=0}^\ell {n \choose k} {\ell \choose k} = {n+\ell \choose \ell}.
$$

\item  In the $x_d$ basis, the irreducible character values are 
$$
\chi^n_k( x_d) = \begin{cases}
1, & \text{if $d$ has exactly $k$ vertical edges and no other edges}, \\
0, &  \text{otherwise}.
\end{cases}
$$

\item The center of $\CC P_n$ has a basis given by the elements
$$
z_\ell = \sum_{a} x_a, \qquad 0 \le \ell \le n,
$$
where the sum is over all diagrams $a$ with exactly $\ell$ vertical edges and no other edges.

\item The character ring $A_n$ of $\CC P_n$ is the algebra with basis given by the $n$ functions $\chi^n_k(\ell) = {\ell \choose k}$, for $0 \le k \le n$.   For $A_n$ these functions have domain $0 \le \ell \le n$.  Since the irreducible characters form a basis of this algebra, we can re-express the product of any two characters in terms of the basis. The structure constants for the character ring come from  the multinomial coefficients in the following polynomial identity  (see \cite[\S 1.4]{Ri})
$$
{x \choose i} {x \choose j} = \sum_{k={\rm max}(i,j)}^{i+j}  {k \choose i+j-k, k-i, k-j} { x \choose k }.
$$
In the application to $A_n$, the upper limit in the sum is taken to be $k=\max(i+j,n)$ since ${x \choose k} = 0$ for $k > n \ge x = \ell$. 
This gives the corresponding decomposition of the tensor product (see  \cite[\S 11]{CR} for an explanation of tensor products)  $V^n_i \otimes V^n_j = \oplus_{k}  {k \choose i+j-k, k-i, k-j} V^n_k$.  This illustrates that even the multiplicative structure of the binomial coefficients is captured in the representation theory of $P_n$. 
 
 \item Some of the first examples of diagram algebras are the group algebra of the symmetric group, with a basis of permutation diagrams, and the Artin braid group, with a basis of braid diagrams.  The Brauer algebra was defined in the 1930s, and its planar version, the Temperley-Lieb algebra, is important in statistical mechanics. The papers \cite{CFS}, \cite{HR}, \cite{GL}, and the references therein, give definitions and examples of these and other diagram algebras.  The planar rook algebra was constructed to be a diagram algebra whose Bratteli diagram is Pascal's triangle.    A good project is to find algebras whose Bratteli diagrams match the lattice of other recursively defined integers such as the Stirling numbers or the trinomial numbers.
 
 \item The planar rook algebra also has representation theoretic importance as the centralizer algebra of the general linear group $G=GL_1(\CC) = \CC \setminus \{0\}$.  Let $V = V_0 \oplus V_1$ such that $V_0
 $ and $V_1$ are the 1-dimensional $G$-modules where $z (v_0 + v_1) = v_0 + z v_1$ for $v_i \in V_i$ and $z \in G$.  Then $\CC P_n  \cong \End_{G}(V^{\otimes n}),$ which is  the  algebra of  all endomorphisms of  the tensor product $V^{\otimes n}$ that commute with $G$ (i.e., $\CC P_n$ is the centralizer of $G$ on $V^{\otimes n}$).   This is analogous to classical Schur-Weyl duality, where the group algebra of the symmetric group $\CC S_n$ is the centralizer of $GL_k(\CC)$ on $W^{\otimes n}$, for $k \ge n$, where $W = \CC^k$ is the representation of $GL_k(\CC)$ by matrix multiplication on column vectors.  If we replace a simple tensor with the subset indexed by  the binary string in its subscripts --- for example $v_1 \otimes v_0 \otimes v_1 \otimes v_1 \otimes v_0 \Leftrightarrow 10110 \Leftrightarrow \{1,3,4\}$  --- then the action on simple tensors is the same as the action of $P_n$ on subsets in Section 2.
 
\end{enumerate}

\bigskip

\noindent {\bf ACKNOWLEDGEMENTS}.  We are grateful to Arun Ram who, at the 2004 AMS sectional meeting in Evanston, helped the second author  discover the planar rook monoid as an example of an algebra whose Bratteli diagram is Pascal's triangle.  We also thank Lex Renner for other important suggestions at that same meeting.

\end{document}